\documentclass{article}
\usepackage{amsthm}
\usepackage{amssymb}
\usepackage{amsmath}
\usepackage{epsfig}
\usepackage{colonequals}
\usepackage{enumerate}
\usepackage{hyperref}
\usepackage{tikz}
\usetikzlibrary{calc}

\newcommand{\vc}[1]{\ensuremath{\vcenter{\hbox{#1}}}}

\newcommand{\ray}{\mathrm{ray}}
\newcommand{\tG}{\tilde{G}}

\newtheorem{theorem}{Theorem}
\newtheorem{problem}[theorem]{Problem}
\newtheorem{corollary}[theorem]{Corollary}
\newtheorem{lemma}[theorem]{Lemma}
\newtheorem{observation}[theorem]{Observation}
\newtheorem{conjecture}[theorem]{Conjecture}

\newcommand{\oururl}{\url{http://lidicky.name/pub/4cone/}}

\newcommand{\makenu}[2]{
 \pgfmathsetmacro{\MB}{360/#1}
  \pgfmathsetmacro{\shift}{#2}
 \pgfmathtruncatemacro{\upto}{#1-1}
 \foreach \i in {0,...,\upto}{
 \draw (\shift+90+\i*\MB:1.3) node{\i};
 }
}

\newcommand{\halfedgesFive}{
\draw[dashed] (0,0) circle(1cm);
\foreach \i in {1,...,5}{\draw (90+72*\i:1.1) coordinate (he\i);}
}

\tikzset{
vtx/.style={inner sep=1.5pt, outer sep=0pt, circle, fill=black,draw}
}

\begin{document}
\title{Coloring count cones of planar graphs}
\author{%
     Zden\v{e}k Dvo\v{r}\'ak\thanks{Computer Science Institute (CSI) of Charles University,
           Malostransk{\'e} n{\'a}m{\v e}st{\'\i} 25, 118 00 Prague, 
           Czech Republic. E-mail: \protect\href{mailto:rakdver@iuuk.mff.cuni.cz}{\protect\nolinkurl{rakdver@iuuk.mff.cuni.cz}}.
           Supported by the Neuron Foundation for Support of Science under Neuron Impuls programme.}
\and	   
Bernard Lidick\'y\thanks{Department of Mathematics, Iowa State University. Ames, IA, USA. E-mail: \protect\href{mailto:lidicky@iasate.edu}{\protect\nolinkurl{lidicky@iastate.edu}}. 
Supported in part by NSF grants DMS-1600390 and DMS-1855653.}
}
\date{\today}
\maketitle
\begin{abstract}
For a plane near-triangulation $G$ with the outer face bounded by a~cycle $C$, let $n^\star_G$ denote
the function that to each $4$-coloring $\psi$ of $C$ assigns the number of ways
$\psi$ extends to a $4$-coloring of $G$.  The block-count reducibility argument
(which has been developed in connection with attempted proofs of
the Four Color Theorem) is equivalent to the statement that the function $n^\star_G$ belongs
to a certain cone in the space of all functions from $4$-colorings of $C$ to real numbers.  We investigate
the properties of this cone for $|C|=5$, formulate a conjecture strengthening
the Four Color Theorem, and present evidence supporting this conjecture.
\end{abstract}


\medskip 

By the Four Color Theorem~\cite{AppHak1,AppHakKoc,rsst}, every planar graph
is $4$-colorable.  Nevertheless, many natural followup questions regarding
$4$-colorability of planar graphs are wide open.  Even very basic precoloring
extension questions, such as the one given in the following problem, are unresolved
(a \emph{near-triangulation} is a connected plane graph in which all faces except for
the outer one have length three).

\begin{problem}\label{prob-ext}
Does there exists a polynomial-time algorithm which, given a near-triangulation $G$ with
the outer face bounded by a $4$-cycle $C$ and a $4$-coloring $\psi$ of $C$,
correctly decides whether $\psi$ extends to a $4$-coloring of $G$?
\end{problem}
Note that there exist infinitely many near-triangulations $G$ with the outer face bounded by a $4$-cycle $C$
such that not every precoloring of $C$ extends to a $4$-coloring of $G$; and we do not have any good guess
at how the near-triangulations with this property could be described.

Nevertheless, we do have some information about the precoloring extension properties
of plane near-triangulations.  For a plane near-triangulation $G$ with the outer face bounded by a cycle $C$, let $n^\star_G$ denote
the function that to each $4$-coloring $\psi$ of $C$ assigns the number of ways
$\psi$ extends to a $4$-coloring of $G$; hence, $\psi$ extends to a $4$-coloring of $G$ if and only if $n^\star_G(\psi)\neq 0$.
Suppose $C=v_1v_2v_3v_4$ is a $4$-cycle and $\psi_1$, $\psi_2$ and $\psi_3$ are its $4$-colorings such that
$\psi_i(v_j)=j$ for $i\in\{1,2,3\}$ and $j\in\{1,2\}$, $\psi_1(v_3)=\psi_3(v_3)=1$, $\psi_2(v_3)=3$, $\psi_1(v_4)=\psi_2(v_4)=2$, and
$\psi_3(v_4)=4$; see Figure~\ref{fig-psi}.
A standard Kempe chain argument shows that if $n^\star_G(\psi_1)\neq 0$, then $n^\star_G(\psi_2)\neq 0$ or $n^\star_G(\psi_3)\neq 0$.

\begin{figure}
\begin{center}
\begin{tikzpicture}
\draw[fill=gray!50!white] (45:1) -- (135:1) -- (225:1) -- (-45:1) -- cycle;
\draw 
(45:1) node[fill=white,circle,draw,label=right:$v_1$]{$1$}
(135:1) node[fill=white,circle,draw,label=left:$v_2$]{$2$}
(225:1) node[fill=white,circle,draw,label=left:$v_3$]{$1$}
(-45:1) node[fill=white,circle,draw,label=right:$v_4$]{$2$}
;
\draw(0,-1.5) node {$\psi_1$};
\end{tikzpicture}
\hskip 1em
\begin{tikzpicture}
\draw[fill=gray!50!white] (45:1) -- (135:1) -- (225:1) -- (-45:1) -- cycle;
\draw 
(45:1) node[fill=white,circle,draw,label=right:$v_1$]{$1$}
(135:1) node[fill=white,circle,draw,label=left:$v_2$]{$2$}
(225:1) node[fill=white,circle,draw,label=left:$v_3$]{$3$}
(-45:1) node[fill=white,circle,draw,label=right:$v_4$]{$2$}
;
\draw(0,-1.5) node {$\psi_2$};
\end{tikzpicture}
\hskip 1em
\begin{tikzpicture}
\draw[fill=gray!50!white] (45:1) -- (135:1) -- (225:1) -- (-45:1) -- cycle;
\draw 
(45:1) node[fill=white,circle,draw,label=right:$v_1$]{$1$}
(135:1) node[fill=white,circle,draw,label=left:$v_2$]{$2$}
(225:1) node[fill=white,circle,draw,label=left:$v_3$]{$1$}
(-45:1) node[fill=white,circle,draw,label=right:$v_4$]{$4$}
;
\draw(0,-1.5) node {$\psi_3$};
\end{tikzpicture}
\end{center}
\caption{Precolorings $\psi_1$, $\psi_2$, and $\psi_3$ of a $4$-cycle.}\label{fig-psi}
\end{figure}

Actually, much more information can be obtained along these lines, using the idea of \emph{Block-count reducibility}~\cite{bcr1,bcr2} developed in
connection with the attempts to prove the Four Color Theorem: Certain inequalities between linear combinations of
$n^\star_G(\psi_1)$, $n^\star_G(\psi_2)$, and $n^\star_G(\psi_3)$ are satisfied for all near-triangulations $G$,
or equivalently, the vector $(n^\star_G(\psi_1),n^\star_G(\psi_2),n^\star_G(\psi_3))$ is contained in a certain cone in $\mathbb{R}^3$.
The main goal of this note is to present and motivate a conjecture regarding this cone in the case of near-triangulations with the
outer face bounded by a $5$-cycle; this conjecture strengthens the Four Color Theorem.  We also provide evidence supporting this conjecture.

\section{Definitions}

In order to describe the cone we alluded to in the introduction, we need a number of definitions, which we introduce in this section.
It is easier to state the idea in the dual setting of $3$-edge-colorings of cubic plane graphs, which is well-known to be
equivalent to $4$-coloring of plane triangulations~\cite{tait}.

Some graphs in this paper may have parallel edges or loops. 
We call two parallel edges a \emph{double edge} and three parallel edges a \emph{triple edge}.

\subsection{Near-cubic graphs and their edge-colorings}

Let $G$ be a connected graph and let $v$ be a vertex of $G$.
 A \emph{half-edge} is $(e,u)$, where $e$ is an edge and $u$ is one of its endpoints. If $e=uv$, when we say $u$ is \emph{incident} with $(e,u)$ but it is not incident with $(e,v)$.
We consider each edge $e=uv$  of $G$ as consisting of two half-edges $(e,u)$ and $(e,v)$ even if $e$ is a loop.
Let $\nu$ be a bijection between the half-edges incident with $v$
and $\{0,\ldots,\deg(v)-1\}$ (so, if $v$ is incident with a loop, each half of the loop is assigned a different number by $\nu$).
If all vertices of $G$ other than $v$ have degree three,
we say that $\tG=(G,v,\nu)$ is a \emph{near-cubic graph}.  We say that $\tG$ is a \emph{plane near-cubic graph} if $G$
is a plane graph and the half-edges incident with $v$ are drawn around it in the clockwise cyclic order $\nu^{-1}(0)$, \ldots, $\nu^{-1}(\deg(v)-1)$.  We define $d(\tG)=\deg(v)$.

A \emph{$3$-edge-coloring} of $\tG$ is an assignment of colors $1$, $2$, and $3$ to edges of $G$ such that
any two edges incident with a common vertex other than $v$ have different colors.  For an integer $d\ge 2$,
a function $\psi:\{0,\ldots,d-1\}\to \{1,2,3\}$
is a \emph{$d$-precoloring} if $|\psi^{-1}(1)|\equiv |\psi^{-1}(2)|\equiv |\psi^{-1}(3)|\equiv d\pmod 2$. This parity condition is necessary, see Observation~\ref{obs:parity}.
We say that a $3$-edge-coloring $\varphi$ of $\tG$ \emph{extends} a $d(\tG)$-precoloring $\psi$ if for any edge $e$ incident with $v$
and a half-edge $h$ of $e$ incident with $v$, we have $\varphi(e)=\psi(\nu(h))$.

\begin{observation}\label{obs:parity}
For an integer $d\ge 2$, if a function $\psi:\{0,\ldots,d-1\}\to \{1,2,3\}$
does not satisfy $|\psi^{-1}(1)|\equiv |\psi^{-1}(2)|\equiv |\psi^{-1}(3)|\equiv d\pmod 2$, then there is no $\tG$ and 3-edge-coloring $\varphi$ of $\tG$ such that $\varphi$ extends $\psi$.
\end{observation}
\begin{proof}
Let $\varphi$ be a 3-edge-coloring of $\tG$ extending $\psi$.
Let $n = |V(\tG)|$. Since $\tG$ is cubic except for one vertex of degree $d$,  the handshaking lemma gives $|E(\tG)| = (3(n-1) + d)/2$.
For $i \in \{1,2,3\}$, then number of edges colored by 
$|\varphi^{-1}(i)|= ((n-1)+|\psi^{-1}(i)|) / 2$.
Therefore the parities of $d$, $|\psi^{-1}(1)|$, $|\psi^{-1}(2)|$, and $|\psi^{-1}(3)|$  are the same.
\end{proof}

Let $n_{\tG}(\psi)$ denote the number of $3$-edge-colorings of $\tG$ which extend $\psi$.
Via the theory of nowhere-zero flows~\cite{tutteflow}, it is easy to establish
the following correspondence between $4$-colorings of near-triangulations and $3$-edge-colorings in their duals. 
Recall  $n^\star_{G}(\psi)$ denotes the number of $4$-colorings of $G$ which extend $\psi$.
\begin{observation}\label{obs-corr}
Let $\tG=(G,v,\nu)$ be a plane near-cubic graph, and let $G^\star$ be the dual of $G$ drawn so that
the outer face of $G^\star$ corresponds to $v$.  Suppose the outer face of $G^\star$ is bounded by a cycle $C$.
Then there exists a mapping $f$ from $4$-colorings of $C$ to $d(\tG)$-precolorings such that
\begin{itemize}
\item $f$ maps exactly four distinct $4$-colorings of $C$ to each $d(\tG)$-precoloring, and
\item every $4$-coloring $\psi$ of $C$ satisfies $n^\star_{G^\star}(\psi)=n_{\tG}(f(\psi))$.
\end{itemize}
\end{observation}

Given two near-cubic graphs $\tG_1=(G_1,v_1,\nu_1)$ and $\tG_2=(G_2,v_2,\nu_2)$ with $\deg(v_1)=\deg(v_2)$,
let $\tG_1\oplus \tG_2$ denote the graph obtained from $G_1$ and $G_2$ by, for $0\le i\le \deg(v_1)-1$,
removing the half-edges $\nu_1^{-1}(i)$ and $\nu_2^{-1}(i)$ and connecting the other halfs of the edges.
Note that $\tG_1\oplus \tG_2$ is a cubic graph, and if $\tG_1$ and $\tG_2$ are plane near-cubic graphs,
then $\tG_1\oplus \tG_2$ is a cubic planar graph.  Observe that the number of $3$-edge-colorings of $\tG_1\oplus \tG_2$ is
\begin{align}
\sum_{\psi} n_{\tG_1}(\psi)n_{\tilde{G_2}}(\psi),\label{eq:nsum}
\end{align}
where the sum goes over all $\deg(v_1)$-precolorings $\psi$.  For any integer $n\ge 3$, let $\tilde{C}_n$ denote the
plane near-cubic graph $(W_n,v,\nu)$, where $W_n$ is the wheel with the central vertex $v$ adjacent to all
vertices of an $n$-cycle; see Figure~\ref{fig:C5}.

\begin{figure}
\begin{center}

\vc{\begin{tikzpicture}
\clip (-2.4,-2) rectangle (1.3,1.5);
\draw (0,-1.5) node{$\tilde{C}_{5}$};
\foreach \i in {1,...,5}{\draw (90+72*\i:0.7) coordinate (he\i);}
\draw[every node/.style={inner sep=1.8pt,fill,circle}]
\foreach \i in {1,...,5}{
(90+72*\i:0.5) node (x\i){} --  (he\i)
}
(x1)--(x2)--(x3)--(x4)--(x5)--(x1)
;
\draw  (-2,0) node[vtx,label=left:$v$](v){}
(he1) to[out=90+72,in=30] (v)
(he2) to[out=90+2*72,in=-50] (v)
(he3) to[out=90+3*72,in=-90,looseness=1.2] (v)
(he4) to[out=90+4*72,in=90,looseness=2] (v)
(he5) to[out=90,in=60] (v)
(he1) node[above]{1}
(he2) node[left]{2}
(he3) node[right]{3}
(he4) node[above]{4}
(he5) node[right]{0}
;
\end{tikzpicture}}
\hskip 2em
\vc{\begin{tikzpicture}
\clip (-2.4,-2) rectangle (1.3,1.5);
\draw (0,-1.5) node{$\tilde{C}_{5}$};
\makenu{5}{0}
\halfedgesFive
\draw[every node/.style={inner sep=1.8pt,fill,circle}]
\foreach \i in {1,...,5}{
(90+72*\i:0.5) node (x\i){} --  (he\i)
}
(x1)--(x2)--(x3)--(x4)--(x5)--(x1)
;\end{tikzpicture}}
\end{center}
\caption{
$(W_5,v,\nu)$ also known as $\tilde{C}_{5}$. Entire graph on the left and partial drawing as we use in the rest of the paper on the right.
}\label{fig:C5}
\end{figure}

\subsection{Signatures and Kempe chains}

The following definition of $d$-signature is will be used to capture the possible parities of 2-edge-colored cycles containing $v$ in a 3-edge-coloring of $\tG=(G,v,\nu)$ distinguished by the half-edges contained in the cycles. In particular,
parity will be $s\in\{-1,1\}$ and the pair of half-edges will be $m$.  
For an integer $d\ge 2$, a \emph{$d$-signature} is a set $S$ of pairs $(m,s)$, where $m$ is an unordered pair of integers in $\{0,\ldots,d-1\}$
and $s\in\{-1,1\}$, satisfying the following conditions:
\begin{itemize}
\item[(i)] for any distinct $(m_1,s_1),(m_2,s_2)\in S$ we have $m_1\cap m_2=\emptyset$, and
\item[(ii)] $S$ does not contain elements $(\{a,b\},s_1)$ and $(\{c,d\},s_2)$ such that $a<c<b<d$.
\end{itemize}
A $d$-precoloring $\psi$ is \emph{compatible} in (distinct) colors $i,j\in\{1,2,3\}$ with a $d$-signature $S$ if 
\begin{itemize}
\item $\psi^{-1}(\{i,j\})=\bigcup_{(m,s)\in S} m$, and
\item for each $(\{a_1,a_2\},s)\in S$, $\psi(a_1)=\psi(a_2)$ holds if and only if $s=-1$.
\end{itemize}

Now, consider a $3$-edge-coloring $\varphi$ of a near-cubic graph $\tG=(G,v,\nu)$.  Each vertex other than $v$ is incident with edges of all three colors.
Hence, for any distinct $i,j\in\{1,2,3\}$, the subgraph $G_{ij}$ of $G$ consisting of edges of colors $i$ or $j$ is a union of pairwise
edge-disjoint cycles, vertex-disjoint except for possible intersections in $v$.
An \emph{$ij$-Kempe chain} of $\varphi$ is a cycle $C$ in $G_{ij}$ containing $v$;
the sign $\sigma(C)$ of the $ij$-Kempe chain $C$ is $1$ if the length of $C$ is even and $-1$ if the length of $C$ is odd.
If $h_1$ and $h_2$ are the half-edges
in $C$ incident with $v$, we let $\mu(C)=\{\nu(h_1),\nu(h_2)\}$.  The \emph{$ij$-Kempe chain signature} $\sigma_{ij}(\varphi)$ of $\varphi$ is
defined as
$$\{(\mu(C),\sigma(C)): \text{$C$ is an $ij$-Kempe chain of $\varphi$}\}.$$
Note that if $\tG$ is plane, then the $ij$-Kempe chains do not cross and the $ij$-Kempe chain signature of $\varphi$
satisfies the condition (ii); and thus $\sigma_{ij}(\varphi)$ is a $d(\tG)$-signature.

\section{Coloring count cones}

Let $\tG=(G,v,\nu)$ be a plane near-cubic graph and let $\psi$ be a $d(\tG)$-precoloring.
Suppose that $\psi$ is compatible (in colors $i,j\in\{1,2,3\}$) with a $d(\tG)$-signature $S$.
We define $n_{\tG,S}(\psi)$ as the number of $3$-edge-colorings $\varphi$ of $\tG$ extending $\psi$ such that
$\sigma_{ij}(\varphi)=S$.  Note that swapping the colors $i$ and $j$ on any set of $ij$-Kempe chains of $\varphi$ results
in another $3$-edge-coloring with the same $ij$-Kempe chain signature.  Furthermore, clearly for any permutation $\pi$ of colors,
we have $n_{\tG,S}(\psi\circ \pi)=n_{\tG,S}(\psi)$.  This establishes bijections implying the following.
\begin{observation}
Let $\tG$ be a plane near-cubic graph and let $S$ be a $d(\tG)$-signature.
Any $d(\tG)$-precolorings $\psi_1$ and $\psi_2$ compatible with $S$ satisfy
$$n_{\tG,S}(\psi_1)=n_{\tG,S}(\psi_2).$$
\end{observation}
Hence, we can define an integer $n_{\tG,S}$ to be equal to $n_{\tG,S}(\psi)$ for an arbitrarily chosen $d(\tG)$-precoloring $\psi$ compatible with $S$.

Let $d\ge 2$ be an integer and let $i,j\in\{1,2,3\}$ be distinct colors.  For a $d$-precoloring $\psi$, let us define $\mathcal{S}_{\psi,ij}$ as the set of $d$-signatures compatible with $\psi$ in colors $ij$.
Since every $3$-edge-coloring of $\tG$ has an $ij$-Kempe chain signature, we have
\begin{equation}\label{eq:main}
n_{\tG}(\psi)=\sum_{S\in\mathcal{S}_{\psi,ij}} n_{\tG,S}(\psi)=\sum_{S\in\mathcal{S}_{\psi,ij}} n_{\tG,S}.
\end{equation}
Let $\mathcal{P}_d$ denote the set of all $d$-precolorings and $\mathcal{S}_d$ the set of all $d$-signatures.
We will work in the vector spaces $\mathbb{R}^{\mathcal{P}_d}$ and $\mathbb{R}^{\mathcal{S}_d}$
with coordinates corresponding to the $d$-precolorings and to the $d$-signatures, respectively.
For each integer $d\ge 2$, the \emph{coloring count cone} $B_d$ is the set of all $x\in\mathbb{R}^{\mathcal{P}_d}$ such that
\begin{itemize}
\item $x(\psi)\ge 0$ for every $d$-precoloring $\psi$, and
\item there exists $y\in \mathbb{R}^{\mathcal{S}_d}$ such that
\begin{itemize}
\item $y(S)\ge 0$ for every $d$-signature $S$, and
\item $x(\psi)=\sum_{S\in\mathcal{S}_{\psi,ij}} y(S)$ for every $d$-precoloring $\psi$ and distinct colors $i,j\in\{1,2,3\}$.
\end{itemize}
\end{itemize}
Note that $B_d$ is indeed a cone, i.e., an unbounded polytope closed under linear combinations with non-negative coefficients.
By (\ref{eq:main}), the vector of precoloring extension counts for any plane near-cubic graph belongs to the corresponding
coloring count cone.
\begin{theorem}\label{thm-cone}
For each plane near-cubic graph $\tG$, we have
$$n_{\tG}\in B_{d(\tG)}.$$
\end{theorem}

Each cone is uniquely determined as the set of non-negative linear combinations of its rays.
For $d\in\{2,3,4,5\}$, the rays of $B_d$ are easy to enumerate by hand or using polytope-manipulation software such as Sage Math or the Parma Polyhedra Library (a program
doing so for $d=5$ can be found at \oururl).
For a near-cubic graph $\tG$ such that $n_{\tG}$ is not the zero function, let
$\ray(\tG)$ denote the set of all non-negative multiples of $n_{\tG}$.
Graphs $\tilde{R}_{2,1},\ldots,\tilde{R}_{5,12}$ used in the following lemma are depicted in Figure~\ref{fig-rays}. 
\begin{lemma}\label{lemma-rays}
Refering to graphs in Figure~\ref{fig-rays}:
\begin{itemize}
\item the cone $B_2$ has exactly one ray equal to $\ray(\tilde{R}_{2,1})$;
\item the cone $B_3$ has exactly one ray equal to $\ray(\tilde{R}_{3,1})$;
\item the cone $B_4$ has exactly four rays equal to $\ray(\tilde{R}_{4,1})$, \ldots, $\ray(\tilde{R}_{4,4})$; and
\item the cone $B_5$ has exactly $12$ rays equal to $\ray(\tilde{R}_{5,1})$, \ldots, $\ray(\tilde{R}_{5,12})$.
\end{itemize}
\end{lemma}
Let us remark that $B_6$ has 208 rays; the direct method we employ is too slow to enumerate all rays for $d\ge 7$ on current workstations.

\begin{figure}
\begin{center}
\vc{\begin{tikzpicture}
\makenu{2}{90}
\draw (0,-1.5) node{$\tilde{R}_{2,1}$};
\draw[dashed] (0,0) circle(1cm);
\draw (180:1.1) -- (0:1.1);
\end{tikzpicture}}
\vc{\begin{tikzpicture}
\makenu{3}{0}
\draw (0,-1.5) node{$\tilde{R}_{3,1}$};
\draw[dashed] (0,0) circle(1cm);
\draw[every node/.style={inner sep=1.8pt,fill,circle}]
(0,0) node(x){} 
(x) -- ++(30:-1.1)
(x) -- ++(150:-1.1)
(x) -- ++(90:1.1)
;
\end{tikzpicture}}

\newcommand{\halfedgesFour}{
\draw[dashed] (0,0) circle(1cm);
\foreach \i in {1,...,4}{\draw (90*\i:1.1) coordinate (he\i);}
}

\foreach \i in {1,2}{
\vc{\begin{tikzpicture}
\draw (0,-1.5) node{$\tilde{R}_{4,\i}$};
\makenu{4}{-45+\i*90}
\begin{scope}[rotate=90*0+45]
\halfedgesFour
\draw[every node/.style={inner sep=1.8pt,fill,circle}]
(he1) to[bend left=45, looseness=1.2] (he2)
(he3) to[bend left=45, looseness=1.2] (he4)
;
\end{scope}
\end{tikzpicture}}
}
\foreach \i in {3,4}{
\vc{\begin{tikzpicture}
\draw (0,-1.5) node{$\tilde{R}_{4,\i}$};
\makenu{4}{45+\i*90}
\begin{scope}[rotate=90*\i+45]
\halfedgesFour
\draw[every node/.style={inner sep=1.8pt,fill,circle}]
(-45:-0.3) node(a){}
(-45:0.3) node(b){}
(he1) -- (a) -- (he2)
(he3) -- (b) -- (he4)
(a)--(b)
;
\end{scope}
\end{tikzpicture}}
}

\foreach \i in {1,...,5}{
\vc{\begin{tikzpicture}
\draw (0,-1.5) node{$\tilde{R}_{5,\i}$};
\makenu{5}{-72+72*\i}
\begin{scope}[rotate=72*1-72]
\halfedgesFive
\draw[every node/.style={inner sep=1.8pt,fill,circle}]
(90:0) node(a){}
(a) -- (he1)
(a) -- (he4)
(a) -- (he5)
(he2) to[bend left=60, looseness=1.5] (he3)
;
\end{scope}
\end{tikzpicture}}
}
\foreach \i in {6,7,8,9,10}{
\vc{\begin{tikzpicture}
\draw (0,-1.5) node{$\tilde{R}_{5,\i}$};
\makenu{5}{-72+72*\i}
\begin{scope}[rotate=72*1-72]
\halfedgesFive
\draw[every node/.style={inner sep=1.8pt,fill,circle}]
(90:0.4) node(a){}  -- (210:0.4) node(b){}  (a) -- (330:0.4) node(c){}
(b) -- (he1)
(b) -- (he2)
(c) -- (he3)
(c) -- (he4)
(a) -- (he5)
;
\end{scope}
\end{tikzpicture}}
}
\vc{\begin{tikzpicture}
\draw (0,-1.5) node{$\tilde{R}_{5,11}$};
\makenu{5}{0}
\halfedgesFive
\draw[every node/.style={inner sep=1.8pt,fill,circle}]
\foreach \i in {1,...,5}{
(90+72*\i:0.5) node (x\i){} --  (he\i)
}
(x1)--(x2)--(x3)--(x4)--(x5)--(x1)
;\end{tikzpicture}}
\vc{\begin{tikzpicture}
\draw (0,-1.5) node{$\tilde{R}_{5,12}$};
\makenu{5}{0}
\halfedgesFive
\draw[every node/.style={inner sep=1.8pt,fill,circle}]
\foreach \i in {1,...,5}{
(90+72*\i:0.5) node (x\i){} -- (he\i)
}
(x1)--(x3)--(x5)--(x2)--(x4)--(x1)
;\end{tikzpicture}}

\end{center}
\caption{Graphs $\tilde{R}_{2,1},\ldots,\tilde{R}_{5,12}$. The dashed circle intersects the half-edges incident with the vertex $v$, which is not depicted for the sake of clarity; the values of $\nu$ are written at the respective half-edges.}\label{fig-rays}
\end{figure}

\section{The cone $B_5$ and the conjecture}

Note that while $\tilde{R}_{5,1}$, \ldots, $\tilde{R}_{5,11}$ are plane, $\tilde{R}_{5,12}$ is not.  Indeed, the following
holds.
\begin{lemma}\label{lem:ab}
The following claims are equivalent.
\begin{itemize}
\item[\textnormal{(a)}] Every planar cubic $2$-edge-connected graph is $3$-edge-colorable.
\item[\textnormal{(b)}] For every plane near-cubic graph $\tG$ with $d(\tG)=5$,
if $n_{\tG}\in\ray(\tilde{R}_{5,12})$, then $n_{\tG}$ is the zero function.
\end{itemize}
\end{lemma}
\begin{proof}
Let us first prove that (a) implies (b).
Consider a plane near-cubic graph $\tG=(G,v,\nu)$ such that $n_{\tG}\in \ray(\tilde{R}_{5,12})$,
and thus for some constant $c\ge 0$, we have $n_{\tG}(\psi)=c\cdot n_{\tilde{R}_{5,12}}(\psi)$ for every $5$-precoloring $\psi$.
Observe that $n_{\tilde{R}_{5,12}}(\psi)n_{\tilde{C}_5}(\psi)=0$ for every $5$-precoloring $\psi$
(since $\tilde{R}_{5,12}\oplus\tilde{C}_5$ is the Petersen graph, which is not $3$-edge-colorable; see Figure~\ref{fig:pet}), and thus
the number of $3$-edge-colorings of $\tG\oplus \tilde{C}_5$ using \eqref{eq:nsum} is
$$\sum_{\psi} n_{\tG}n_{\tilde{C}_5}(\psi)=c\sum_{\psi} n_{\tilde{R}_{5,12}}n_{\tilde{C}_5}(\psi)=0.$$
Hence, the planar cubic graph $\tG\oplus \tilde{C}_5$ is not $3$-edge-colorable.
By (a), $\tG\oplus \tilde{C}_5$ has a bridge, and thus $G$ has a bridge.  But then a standard parity argument implies
that $\tG$ has no $3$-edge-coloring, and thus $n_{\tG}$ is the zero function.

\begin{figure}
\begin{center}
\vc{\begin{tikzpicture}
\draw (0,-1.8) node{$\tilde{R}_{5,12}$};
\draw[every node/.style={inner sep=1.8pt,fill,circle}]
\foreach \i in {1,...,5}{
(90+72*\i:0.5) node (x\i){} 
(3,0)++(90+72*\i:0.5) node (y\i){}
}
(x1)--(x3)--(x5)--(x2)--(x4)--(x1)
(y1)--(y2)--(y3)--(y4)--(y5)--(y1)
;
\draw (3,-1.8) node{$\tilde{C}_{5}$};
\draw  

(x5) to[bend left=90] node[pos=0.1,label=left:0]{}  node[pos=0.9,label=left:0]{} (y5)

(x1) to[out=120,in=120,looseness=1.8] node[pos=0.1,label=left:1]{}  node[pos=0.9,label=left:1]{} (y1)

(x4) to[out=60,in=60,looseness=1.8] node[pos=0.1,label=left:4]{}  node[pos=0.9,label=left:4]{} (y4)

(x2) to[out=270,in=270,looseness=1] node[pos=0.1,label=left:2]{}  node[pos=0.9,label=left:2]{} (y2)

(x3) to[out=270,in=270,looseness=1] node[pos=0.1,label=right:3]{}  node[pos=0.9,label=right:3]{} (y3)
;
\end{tikzpicture}}
\hskip 2em
\vc{\begin{tikzpicture}
\draw (0,-1) node{$\tilde{R}_{5,12}$};
\draw[every node/.style={inner sep=1.8pt,fill,circle}]
\foreach \i in {1,...,5}{
(90+72*\i:0.5) node (x\i){} 
(90+72*\i:1.8) node (y\i){} 
}
(x1)--(x3)--(x5)--(x2)--(x4)--(x1)
(y1)--(y2)--(y3)--(y4)--(y5)--(y1)
;
\draw
(x1)-- node[pos=0.3,label=above:1]{} (y1) 
(x2)-- node[pos=0.6,label=above:2]{} (y2) 
(x3)-- node[pos=0.6,label=above:3]{} (y3) 
(x4)-- node[pos=0.3,label=above:4]{} (y4) 
(x5)-- node[pos=0.5,label=right:0]{} (y5) 
;
\draw (0,-1.8) node{$\tilde{C}_{5}$};
\end{tikzpicture}}
\caption{$\tilde{R}_{5,12}\oplus\tilde{C}_5$ in two different drawings.}\label{fig:pet}
\end{center}
\end{figure}

Next, let us prove that (b) implies (a).  Suppose for a contradiction that (b) holds, but there exists a plane cubic $2$-edge-connected graph
that is not $3$-edge-colorable, and let $H$ be one with the smallest number of vertices.  By Euler's formula and possible parallel edges,
$H$ has a face $f$ of length $2 \le d\le 5$; hence, we can write $H=\tG\oplus\tilde{C}_d$ for a plane near-cubic graph $\tG$.
By Theorem~\ref{thm-cone}, we have $n_{\tG}\in B_d$, and by Lemma~\ref{lemma-rays}, there
exist non-negative real numbers $c_i$ such that
$$n_{\tG}=\sum_i c_in_{\tilde{R}_{d,i}}.$$
Observe there exists a plane near-cubic graph $\tilde{P}$ with $d-1$ vertices such that $\tG\oplus\tilde{P}$ is $2$-edge-connected.
By the minimality of $H$, $\tG\oplus\tilde{P}$ is $3$-edge-colorable, and in particular $n_{\tG}$ is not the zero function.
By (b), $n_{\tG}$ is not a positive multiple of $n_{\tilde{R}_{5,12}}$, and thus there exists an index $k\le 11$ such that $c_k>0$.
Observe that $\tilde{R}_{d,k}\oplus\tilde{C}_d$ is $3$-edge-colorable, and thus there exists a $d$-precoloring $\psi_0$ such that
$n_{\tilde{R}_{d,k}}(\psi_0)n_{\tilde{C}_d}(\psi_0)>0$.  However, then the number of $3$-edge-colorings of $H$ is
$$\sum_{\psi}n_{\tG}(\psi)n_{\tilde{C}_d}(\psi)\ge c_k\sum_{\psi}n_{\tilde{R}_{d,k}}(\psi)n_{\tilde{C}_d}(\psi)\ge c_kn_{\tilde{R}_{d,k}}(\psi_0)n_{\tilde{C}_d}(\psi_0)>0.$$
This contradicts the assumption that $H$ is not $3$-edge-colorable.
\end{proof}

Note that (a) from Lemma~\ref{lem:ab} is well-known to be equivalent to the Four Color Theorem~\cite{tait}, and thus indeed there is
no plane near-cubic graph $\tG$ with $d(\tG)=5$ such that $n_{\tG}$ is not the zero function
and $n_{\tG}\in \ray(\tilde{R}_{5,12})$; and furthermore, a direct proof of this fact
would imply the Four Color Theorem.  Motivated by this observation (and experimental evidence), we propose the following
conjecture, a strengthening of the Four Color Theorem.  Let $B'_5$ denote the cone in $\mathbb{R}^{\mathcal{P}_d}$
with rays $\ray(\tilde{R}_{5,1})$, \ldots, $\ray(\tilde{R}_{5,11})$.
\begin{conjecture}\label{conj-main}
Every plane near-cubic graph $\tG$ with $d(\tG)=5$ satisfies
$n_{\tG}\in B'_5$.
\end{conjecture}
For $i\in \{0,\ldots, 4\}$, let $\psi^{5,a}_i$ and $\psi^{5,b}_i$ denote the $5$-precolorings whose values at $j\in \{0,\ldots, 4\}$
are defined by the following table; see also Figure~\ref{fig-psi-ab}.
Notice that $i$ is a rotating the coloring.

\begin{center}
\begin{tabular}{c|c|c}
$(j-i)\bmod 5$&$\psi^{5,a}_i(j)$&$\psi^{5,b}_i(j)$\\
\hline
$0$ & $1$ & $1$\\
$1$ & $1$ & $2$\\
$2$ & $2$ & $1$\\
$3$ & $3$ & $1$\\
$4$ & $1$ & $3$
\end{tabular}
\end{center}

\begin{figure}
\begin{center}
\vc{\begin{tikzpicture}
\draw (0,-1.5) node{$\psi_0^{5,a}$};
 \pgfmathsetmacro{\MB}{360/5}
 \foreach \i/\c in {0/1,1/1,2/2,3/3,4/1}{
 \draw (90+\i*\MB:1.3) node{\i};
 \draw (90+\i*\MB+15:0.8) node{\c};
}
\draw[dashed] (0,0) circle(1cm);
\draw\foreach \i in {1,...,5}{(90+72*\i:0.1) --  (90+72*\i:1.1)};
\draw[fill=gray!30] (0,0) circle(0.5cm);
\draw (0,0) node{ };
\end{tikzpicture}}
\hskip 2em
\vc{\begin{tikzpicture}
\draw (0,-1.5) node{$\psi_1^{5,a}$};
 \pgfmathsetmacro{\MB}{360/5}
 \foreach \i/\c in {0/1,1/1,2/1,3/2,4/3}{
 \draw (90+\i*\MB:1.3) node{\i};
 \draw (90+\i*\MB+15:0.8) node{\c};
}
\draw[dashed] (0,0) circle(1cm);
\draw\foreach \i in {1,...,5}{(90+72*\i:0.1) --  (90+72*\i:1.1)};
\draw[fill=gray!30] (0,0) circle(0.5cm);
\draw (0,0) node{ };
\end{tikzpicture}}
\hskip 2em
\vc{\begin{tikzpicture}
\draw (0,-1.5) node{$\psi_0^{5,b}$};
 \pgfmathsetmacro{\MB}{360/5}
 \foreach \i/\c in {0/1,1/2,2/1,3/1,4/3}{
 \draw (90+\i*\MB:1.3) node{\i};
 \draw (90+\i*\MB+15:0.8) node{\c};
}
\draw[dashed] (0,0) circle(1cm);
\draw\foreach \i in {1,...,5}{(90+72*\i:0.1) --  (90+72*\i:1.1)};
\draw[fill=gray!30] (0,0) circle(0.5cm);
\draw (0,0) node{ };
\end{tikzpicture}}
\end{center}
\caption{Precolorings $\psi_0^{5,a}$ and $\psi_0^{5,b}$.}\label{fig-psi-ab}
\end{figure}

Note that each $5$-precoloring is obtained from one of these ten by a permutation of colors.
The cone $B'_5$ has exactly one facet which is not also a facet of $B_5$, giving an equivalent formulation of Conjecture~\ref{conj-main}.
\begin{conjecture}
Every plane near-cubic graph $\tG$ with $d(\tG)=5$ satisfies
$$3\sum_{i=0}^4 n_{\tG}(\psi^{5,a}_i)\ge \sum_{i=0}^4 n_{\tG}(\psi^{5,b}_i).$$
\end{conjecture}

In the rest of the note, we provide some evidence supporting Conjecture~\ref{conj-main}; in particular, we show there are no counterexamples to the conjecture for plane near-cubic graphs
with less than 30~vertices.

\section{Evidence}

In this section we  present experimental evidence for the validity of Conjecture~\ref{conj-main}. 
Our goal is to show Corollary~\ref{cor-cexsize} stating that Conjecture~\ref{conj-main} holds for graphs on at most 30 vertices.
The main idea of our approach is to generate larger graphs $\tG$ from smaller graphs by planarity preserving operations. One such is depicted in Figure~\ref{fig-gamma}. We will generate ``all'' possibilities for $d(\tG)\leq 7$ and particular ones with $d(\tG) = 8$. We then argue that all graphs in at most 30 vertices can be generated this way.  

We begin by stating a few more definitions.
A vector $x\in \mathcal{P}_d$ is \emph{invariant with respect to permutation of colors} if all $d$-precolorings $\psi$ and $\psi'$ that only differ by a permutation of colors
satisfy $x(\psi)=x(\psi')$.

See Figure~\ref{fig:rotflip} for an illustration of the following definitions.  The \emph{rotation by $t$} of a $d$-precoloring $\psi$
is the $d$-precoloring $r_t(\psi)$ such that $r_t(\psi)((i+t)\bmod d)=\psi(i)$ for $i\in \{0,\ldots, d-1\}$.  The \emph{flip} of a $d$-precoloring $\psi$
is the $d$-precoloring $f(\psi)$ such that $f(\psi)(i)=\psi(d-1-i)$ for $i\in \{0,\ldots, d-1\}$.
For $x\in \mathbb{R}^{\mathcal{P}_d}$, let $r_t(x)$ be defined as $y\in \mathbb{R}^{\mathcal{P}_d}$ such that $y(r_t(\psi))=x(\psi)$ for every $d$-precoloring $\psi$,
and let $f(x)$ be defined as $z\in \mathbb{R}^{\mathcal{P}_d}$ such that $z(f(\psi))=x(\psi)$ for every $d$-precoloring $\psi$.
A set $K\subseteq \mathbb{R}^{\mathcal{P}_d}$ is \emph{closed under rotations and flips} if we have $x\in K$ if and only if $f(x)\in K$ and $r_t(x)\in K$ for all $t\in \{0,1,\ldots,d-1\}$.
For a near-cubic graph $\tG=(G,v,\nu)$ with $\deg(v)=d$, let $r_t(\tG)$ denote the near-cubic graph $(G,v,\nu_1)$, where $\nu_1^{-1}((i+t)\bmod d)=\nu^{-1}(i)$ for $i\in \{0,\ldots, d-1\}$,
and let $f(\tG)$ denote the near-cubic graph $(G,v,\nu_2)$, where $\nu_2^{-1}(i)=\nu_2^{-1}(d-1-i)$ for $i\in \{0,\ldots, d-1\}$.
\begin{observation}
Let $\tG$ be a near-cubic graph, $d=d(\tG)$ and $t\in \{0,\ldots, d-1\}$.
Then $n_{r_t(\tG)}=r_t(n_{\tG})$ and $n_{f(\tG)}=f(n_{\tG})$.
\end{observation}

\begin{figure}
\begin{center}
\begin{tikzpicture}
\draw(0,0) node{
$\psi$\vc{
\setlength{\tabcolsep}{2pt}
\begin{tabular}{c|c|c|c}
0 & 1 & $\cdots$ & $d-1$ \\ \hline
$c_0$ & $c_1$  & $\cdots$ & $c_{d-1}$
\end{tabular}
}}

(4,0) node{
$r_1(\psi)$\vc{
\setlength{\tabcolsep}{2pt}
\begin{tabular}{c|c|c|c}
0 & 1 & $\cdots$ & $d-1$ \\ \hline
$c_{d-1}$  & $c_0$ & $\cdots$ & $c_{d-2}$
\end{tabular}
}
}

(8,0) node{
$f(\psi)$\vc{
\setlength{\tabcolsep}{2pt}
\begin{tabular}{c|c|c|c}
0 & 1 &  $\cdots$ & $d-1$ \\ \hline
$c_{d-1}$ & $c_{d-2}$ & $\cdots$ & $c_{0}$
\end{tabular}
}}
;
\end{tikzpicture}

\vc{\begin{tikzpicture}
\draw (0,-1.5) node{$\tilde{G}$};
 \pgfmathsetmacro{\MB}{360/5}
 \foreach \i in {0,1,2,3,4}{\draw (90+\i*\MB:1.3) node{\i};} 
\draw[dashed] (0,0) circle(1cm);
\draw\foreach \i in {1,...,5}{(90+72*\i:0.5) --  (90+72*\i:1.1)};
\draw[fill=gray!30] (0,0) circle(0.6cm);
\draw (0,0) node{$G-v$};
\end{tikzpicture}}
\vc{\begin{tikzpicture}
\draw (0,-1.5) node{$r_1(\tilde{G})$};
 \pgfmathsetmacro{\MB}{360/5}
 \foreach \i in {0,1,2,3,4}{\draw (90+\i*\MB-\MB:1.3) node{\i};} 
\draw[dashed] (0,0) circle(1cm);
\draw\foreach \i in {1,...,5}{(90+72*\i:0.5) --  (90+72*\i:1.1)};
\draw[fill=gray!30] (0,0) circle(0.6cm);
\draw (0,0) node{$G-v$};
\end{tikzpicture}}
\vc{\begin{tikzpicture}
\draw (0,-1.5) node{$f(\tilde{G})$};
 \pgfmathsetmacro{\MB}{360/5}
 \foreach \i in {0,1,2,3,4}{\draw (90-\i*\MB-\MB:1.3) node{\i};} 
\draw[dashed] (0,0) circle(1cm);
\draw\foreach \i in {1,...,5}{(90+72*\i:0.5) --  (90+72*\i:1.1)};
\draw[fill=gray!30] (0,0) circle(0.6cm);
\draw (0,0) node{$G-v$};
\end{tikzpicture}}
\end{center}
\caption{Rotation and flip operations. Colors are denoted by $c_0,\ldots,c_{d-1}$.}\label{fig:rotflip}
\end{figure}

Let $\psi_1$ be a $d_1$-precoloring and $\psi_2$ a $d_2$-precoloring.  For an integer $k\le \min(d_1,d_2)$, we say that $\psi_1$ \emph{$k$-matches} $\psi_2$ if
$\psi_1(d_1-k+i)=\psi_2(d_2-1-i)$ for $i\in \{0,1,\ldots, k-1\}$.  By $\gamma_k(\psi_1,\psi_2)$, we denote the $(d_1+d_2-2k)$-precoloring $\gamma$
such that $\gamma(i)=\psi_1(i)$ for $i\in \{0,\ldots, d_1-k-1\}$ and $\gamma(i)=\psi_2(i-(d_1-k))$ for $i\in \{d_1-k,\ldots, d_1+d_2-2k-1\}$.
For $x_1\in \mathbb{R}^{\mathcal{P}_{d_1}}$ and $x_2\in \mathbb{R}^{\mathcal{P}_{d_2}}$, we define $\gamma_k(x_1,x_2)$ as the vector $y\in \mathbb{R}^{\mathcal{P}_{d_1+d_2-2k}}$
such that $$y(\psi)=\sum_{\psi_1,\psi_2:\gamma_k(\psi_1,\psi_2)=\psi} x_1(\psi_1)x_2(\psi_2),$$
where the sum is over all $k$-matching $d_1$-precolorings $\psi_1$ and $d_2$-precolorings $\psi_2$.
For near-cubic graphs $\tG_1=(G_1,v_1,\nu_1)$ with $\deg(v_1)=d_1$ and $\tG_2=(G_2,v_2,\nu_2)$ with $\deg(v_2)=d_2$, let
$\gamma_k(\tG_1,\tG_2)$ denote the near-cubic graph $(G,v,\nu)$, where $G$ is obtained from $G_1$ and $G_2$ by identifying $v_1$ with $v_2$ to a single vertex $v$
and for $i\in \{0,1,\ldots, k-1\}$ removing the half-edges $\nu_1^{-1}(d_1-k+i)$ and $\nu_2^{-1}(d_2-1-i)$ and connecting the other halfs of the edges; and
$\nu^{-1}(i)=\nu_1^{-1}(i)$ for $i\in \{0,\ldots, d_1-k-1\}$ and $\nu^{-1}(i)=\nu^{-1}_2(i-(d_1-k))$ for $i\in \{d_1-k,\ldots, d_1+d_2-2k-1\}$. See Figure~\ref{fig-gamma} for an illustration.
\begin{observation}
Let $\tG_1$ and $\tG_2$ be near-cubic graphs.  For every integer $k\in\{0,\ldots, \min(d(\tG_1),d(\tG_2))\}$,
we have $n_{\gamma_k(\tG_1,\tG_2)}=\gamma_k(n_{\tG_1},n_{\tG_2})$.
\end{observation}

\begin{figure}
\begin{center}
\vc{\begin{tikzpicture}
\draw (0,-1.5) node{$\tilde{G}_1$};
 \pgfmathsetmacro{\MB}{360/5}
 \foreach \i in {0,1,2,3,4}{\draw (90+\i*\MB:1.3) node{\i};} 
\draw[dashed] (0,0) circle(1cm);
\draw\foreach \i in {1,...,5}{(90+72*\i:0.5) --  (90+72*\i:1.1)};
\draw[fill=gray!30] (0,0) circle(0.7cm);
\draw (0,0) node{$G_1-v_1$};
\end{tikzpicture}}
\vc{\begin{tikzpicture}
\draw (0,-1.5) node{$\tilde{G}_2$};
 \pgfmathsetmacro{\MB}{360/5}
 \foreach \i in {0,1,2,3,4}{\draw (90+\i*\MB-2*\MB:1.3) node{\i};} 
\draw[dashed] (0,0) circle(1cm);
\draw\foreach \i in {1,...,5}{(90+72*\i:0.5) --  (90+72*\i:1.1)};
\draw[fill=gray!30] (0,0) circle(0.7cm);
\draw (0,0) node{$G_2-v_2$};
\end{tikzpicture}}
\vc{\begin{tikzpicture}
\draw (1.1,-1.5) node{$\gamma_2(\tG_1,\tG_2)$};
 \pgfmathsetmacro{\MB}{360/5}
 \foreach \i in {0,1,2}{\draw (90+\i*\MB:1.3) node{\i};} 
\draw\foreach \i in {1,...,5}{(90+72*\i:0.5) --  (90+72*\i:1.1) coordinate (z\i) };
\draw[fill=gray!30] (0,0) circle(0.7cm);
\draw (0,0) node{$G_1-v_1$};
\begin{scope}[xshift=2.2cm]
 \foreach \i in {3,4,5}{\draw (90+\i*\MB:1.3) node{\i};} 
\draw\foreach \i in {1,...,5}{(90+72*\i:0.5) --  (90+72*\i:1.1) coordinate (y\i) };
\draw[fill=gray!30] (0,0) circle(0.7cm);
\draw (0,0) node{$G_2-v_2$};
\end{scope}
\draw(z4) -- (y1);
\draw(z3) -- (y2);
\draw[dashed,radius=1] (0,1) arc[start angle=90, end angle=270] -- (2.2,-1)  arc[start angle=-90, end angle=90] -- cycle;
\end{tikzpicture}}
\end{center}
\caption{$\gamma_k(\tG_1,\tG_2)$}\label{fig-gamma}
\end{figure}

By a computer-assisted enumeration, we verified the following claim.
\begin{lemma}\label{lemma-sets}
There exists cones $K_d\subseteq \mathbb{R}^{\mathcal{P}_d}$ for $d=2,\ldots,8$ such that the following claims hold.
\begin{itemize}
\item[\textnormal{(a)}] $K_d=B_d$ when $d\le 4$ and $K_5=B'_5$.
\item[\textnormal{(b)}] For all $d\in \{2,\ldots,8\}$, the elements of $K_d$ are invariant with respect to permutation of colors.
\item[\textnormal{(c)}] For $d\in\{2,\ldots,7\}$, the cone $K_d$ is closed under rotations and flips.
\item[\textnormal{(d)}] If $2\le d_1\le d_2$ and $d_1+d_2\le 7$, then for all $x_1\in K_{d_1}$ and $x_2\in K_{d_2}$ we have $\gamma_0(x_1,x_2)\in K_{d_1+d_2}$.
\item[\textnormal{(e)}] If $2\le d\le 5$, then for all $x\in K_d$ we have $\gamma_1(n_{\tilde{R}_{3,1}},x)\in K_{d+1}$.
\item[\textnormal{(f)}] If $3\le d\le 7$, then for all $x\in K_d$ we have $\gamma_2(n_{\tilde{R}_{3,1}},x)\in K_{d-1}$.
\item[\textnormal{(g)}] If $2\le d_1\le 6$ and $1\le c\le d_1/2$, then for all $x_1\in K_{d_1}$ and $x_2\in K_{7+2c-d_1}$, we have $\gamma_c(x_1,x_2)\in K_7$.
\item[\textnormal{(h)}] For every $x_1\in K_8$ and $x_2\in K_7$, we have $\gamma_4(x_1,x_2)\in K_7$.
\item[\textnormal{(i)}] For every $x_1,x_2\in K_6$, we have $r_2(\gamma_2(x_1,x_2))\in K_8$.
\end{itemize}
\end{lemma}
\begin{proof}
The proof and the program to verify the proof can be found at \oururl.
The cones are described by their rays, enumerated in the file.
Cone $K_6$ has 102 rays, $K_7$ has 22605 rays, and $K_8$ has 4330 rays.
It suffices to verify all the claims for $x$, $x_1$, $x_2$ being the rays of the cones specified in the claims; the inclusion
of the resulting vectors in the appropriate cone is certified by expressing them as a linear non-negative combination of the rays of the cone.
\end{proof}

Parts (e) and (f) of Lemma~\ref{lemma-sets} have the following corollary.
\begin{lemma}\label{lemma-growincl}
Let $\tG=(G,v,\nu)$ be a plane near-cubic graph and let $d=d(\tG)$.  If $d\in\{2,\ldots,7\}$
and $n_{\tG}\not\in K_d$, then there exists a plane near-cubic graph $\tG_0=(G_0,v_0,\nu_0)$
such that $d(\tG_0)=7$, $n_{\tG_0}\not\in K_7$, $G_0-v_0$ is an induced subgraph of $G-v$,
and $|V(G_0)|\le |V(G)|-(7-d)$.
\end{lemma}
\begin{proof}
We prove the claim by induction on the number of vertices of $G$.
When $d\le 4$, the claim is vacuously true by Theorem~\ref{thm-cone}, since $K_d=B_d$.
When $d=7$, we can set $\tG_0=\tG$.
Hence, suppose that $d\in\{5,6\}$.  Since $n_{\tG}\not\in K_d$, the function $n_{\tG}$ is not identically zero.

If $G-v$ is disconnected, we can by symmetry assume that
$\tG=\gamma_0(\tG_1,\tG_2)$ for plane near-cubic graphs $\tG_1$ and $\tG_2$ such that
$d=d(\tG_1)+d(\tG_2)$ and $d(\tG_1)\le d(\tG_2)$.  Since $n_{\tG}$ is not the zero function,
$n_{\tG_1}$ is not the zero function either, and thus $d(\tG_1)\neq 1$.
Hence $d(\tG_1)\ge 2$, and thus $2 \le d(\tG_2)\le 4$.  Hence by Lemma~\ref{lemma-sets}(a), $n_{\tG_1}\in K_{d(\tG_1)}$ and $n_{\tG_2}\in K_{d(\tG_2)}$,
and $n_{\tG}\in K_d$ by Lemma~\ref{lemma-sets}(d), which is a contradiction.

Hence, $G-v$ is connected (and the same argument as for disconnected $G-v$  shows that no loop is incident with $v$).
Consequently, $v$ is not incident with a triple edge.
If $v$ is incident with a double edge, then we can by symmetry assume that
$\tG=\gamma_1(\tilde{R}_{3,1},\tG_1)$ for a plane near-cubic graph $\tG_1=(G_1,v_1,\nu_1)$ with $d(\tG_1)=d-1\le 5$.
By Lemma~\ref{lemma-sets}(e), since $n_{\tG}\not\in K_d$, we have $n_{\tG_1}\not\in K_{d-1}$.
By the induction hypothesis, there exists a plane near-cubic graph $\tG_0=(G_0,v_0,\nu_0)$ with $d(\tG_0)=7$,
such that $n_{\tG_0}\not\in K_7$, $G_0-v_0$ is an induced subgraph of $G_1-v_1$, and thus also of $G-v$,
and $|V(G_0)|\le |V(G_1)|-(7-(d-1))<|V(G)|-(7-d)$, as required.

Hence, we can assume $v$ is not incident with a double edge.  Consequently, we can by symmetry assume that
$\tG=\gamma_2(\tilde{R}_{3,1},\tG_1)$ for a plane near-cubic graph $\tG_1=(G_1,v_1,\nu_1)$ with $d(\tG_1)=d+1$.
By Lemma~\ref{lemma-sets}(f), since $n_{\tG}\not\in K_d$, we have $n_{\tG_1}\not\in K_{d+1}$.
By the induction hypothesis, there exists a plane near-cubic graph $\tG_0=(G_0,v_0,\nu_0)$ with $d(\tG_0)=7$,
such that $n_{\tG_0}\not\in K_7$, $G_0-v_0$ is an induced subgraph of $G_1-v_1$, and $|V(G_0)|\le |V(G_1)|-(7-(d+1))=|V(G)|-(7-d)$.
Hence, the claim of the lemma follows.
\end{proof}

We will say that a plane near-cubic graph $\tG=(G,v,\nu)$ is \emph{extremal} if $d(\tG)=7$, $n_{\tG}\not\in K_7$,
and there does not exist any plane near-cubic graph $\tG_0=(G_0,v_0,\nu)$ with $d(\tG_0)=7$ such that $n_{\tG_0}\not\in K_7$
and $G_0-v_0$ is a proper minor of $G-v$.
\begin{lemma}\label{lemma-nosg}
If $\tG=(G,v,\nu)$ is an extremal plane near-cubic graph and $\tG'=(G',v',\nu')$ is a plane near-cubic graph
with $d(\tG')\le 7$ such that $G'-v'$ is a proper minor of $G-v$, then $n_{\tG'}\in K_{d(\tG')}$.
\end{lemma}
\begin{proof}
If $n_{\tG'}\not\in K_{d(\tG')}$, then by Lemma~\ref{lemma-growincl} there would exist a plane near-cubic graph $\tG_0=(G_0,v_0,\nu_0)$
such that $d(\tG_0)=7$, $n_{\tG_0}\not\in K_7$ and $G_0-v_0$ is an induced subgraph of $G'-v'$.
However, then $G_0-v_0$ would be a proper minor of $G-v$, contradicting the assumption that $\tG$ is extremal.
\end{proof}

Next, let us explore consequences of part (g) of Lemma~\ref{lemma-sets}.

\begin{lemma}\label{lemma-2conn}
If $\tG=(G,v,\nu)$ is an extremal plane near-cubic graph, then
$v$ is not incident with loops or parallel edges and $G-v$ is $2$-edge-connected.
\end{lemma}
\begin{proof}
Analogously to the proof of Lemma~\ref{lemma-growincl}, if $v$ were incident with a loop or a parallel edge or if $G-v$ were
not $2$-edge-connected, we would have $\tG=\gamma_c(\tG_1,\tG_2)$ for plane near-cubic graphs $\tG_1$ and $\tG_2$ such that
$2\le d(\tG_1)\le d(\tG_2)$, $d(\tG_1)+d(\tG_2)=7+2c$, and $c\le 1$; in particular, $d(\tG_2)\le 7$
and $d(\tG_1)\le \lfloor (7+2c)/2\rfloor \le 4$.
By Lemma~\ref{lemma-nosg}, we have $n_{\tG_i}\in K_{d(\tG_i)}$ for $i\in \{1,2\}$.
By Lemma~\ref{lemma-sets}(g), we conclude $n_{\tG}\in K_7$, which is a contradiction.
\end{proof}

Suppose $A$ and $B$ form a partition of the vertex set of a graph $H$, and let $S$ be the set
of edges of $H$ with one end in $A$ and the other end in $B$.  In this situation, we say
$S$ is an \emph{edge cut} of $H$ with sides $A$ and $B$.

\begin{lemma}\label{lemma-nosep3}
If $\tG=(G,v,\nu)$ is an extremal plane near-cubic graph, then
$G-v$ does not contain an edge cut $S$ such that $v$ has at least $|S|$ neighbors in each side of the cut.
\end{lemma}
\begin{proof}
Suppose for a contradiction $G-v$ contains such an edge cut $S$ of size~$c$, and thus
$\tG=\gamma_c(\tG_1,\tG_2)$ for plane near-cubic graphs $\tG_1$ and $\tG_2$ such that
$2c\le d(\tG_1)\le d(\tG_2)$ and $d(\tG_1)+d(\tG_2)=7+2c$.  Since $v$ has $7$ neighbors
and at least $c$ of them are contained in each of the sides of the cut, we have $c\le 3$.
Note that $d(\tG_2)\le 7$ and $d(\tG_1)\le \lfloor (7+2c)/2\rfloor \le 6$.
By Lemma~\ref{lemma-nosg}, we have $n_{\tG_i}\in K_{d(\tG_i)}$ for $i\in\{1,2\}$.
By Lemma~\ref{lemma-sets}(g), we conclude $n_{\tG}\in K_7$, which is a contradiction.
\end{proof}

An edge cut $S$ of size at most five in a near-cubic graph $\tG=(G,v,\nu)$ is \emph{essential} if
the side of $S$ containing $v$ contains at least one other vertex and the other side $B$ of $S$ induces neither a tree nor a $5$-cycle.
\begin{lemma}\label{lemma-nones}
If $\tG=(G,v,\nu)$ is an extremal plane near-cubic graph, then
$\tG$ does not contain an essential edge cut $S$ of size at most five.
\end{lemma}
\begin{proof}
Suppose for a contradiction that $\tG$ contains an essential edge-cut $S$ of size $k\le 5$,
and choose one with minimum $k$, and subject to that one for which the side $B$ not containing $v$ is minimal.
We claim $G[B]$ is $2$-edge-connected. Otherwise, $B$ is a disjoint union of non-empty
sets $B_1$ and $B_2$, where $G$ contains $r\le 1$ edges with one end in $B_1$ and the other end in $B_2$.
For $i\in\{1,2\}$, let $S_i$ denote the set of edges of $G$ with exactly one end in $B_i$.
Since $\tG$ is extremal, $n_{\tG}\not\in K_7$ is not identically zero, and thus $G$ is $2$-edge-connected,
implying $|S_i|\ge 2$.  Hence, $|S_i|=k+2r-|S_{3-i}|\le k$.  By the minimality of $B$, we conclude that
$B_i$ induces a tree or a $5$-cycle, and thus $|S_i|\ge 3$.  Hence $5\ge k=|S_1|+|S_2|-2r\ge 6-2r$, and thus
$r=1$ and $|S_1|,|S_2|\le 4$.  This implies that neither $B_1$ nor $B_2$ induces a $5$-cycle, and
thus both of them induce trees; and $G$ contains an edge between them, implying that $B$ induces a tree,
contrary to the assumption that $S$ is an essential edge cut.

Since $G[B]$ is $2$-edge-connected and subcubic, each face of $G[B]$ is bounded by a cycle.
Let $C_S$ denote the cycle bounding the face $f$ of $G[B]$ whose interior contains $v$.
Observe that all edges of $S$ are drawn inside $f$.  Otherwise, the set $S'$ of edges of $S$ drawn inside
$C$ forms an edge cut of order smaller than $k$ and by the minimality of $k$, its side $B'\supsetneq B$
induces a tree or a $5$-cycle; this is not possible, since $G[B]$ is $2$-edge connected
and not a tree.

Let $\tG_c$ be the plane near-cubic graph obtained from $G$ by contracting the side of the cut
containing $v$ to a single vertex. By Lemma~\ref{lemma-nosg}, we have
$n_{\tG_c}\in K_k$.  Since $K_d=B_d$ for $d\le 4$ and $K_5=B'_5$,
$$n_{\tG_c}=\sum_i c_in_{\tilde{R}_{k,i}},$$
where $i\le 11$ if $k=5$ and the coefficients $c_i$ are non-negative.
Let $\tG_i=(G_i,v_i,\nu_i)$ denote the plane near-cubic graph obtained from $\tG$ by replacing the side of the cut $S$ not containing $v$
by $\tilde{R}_{k,i}$.
Note that $n_{\tG}=\sum_i c_in_{\tG_i}$, and since $K_7$ is a cone and $n_{\tG}\not\in K_7$,
there exists $i$ such that $n_{\tG_i}\not\in K_7$.  
Because $B$ contains the cycle $C_S$ and all edges of $S$ are incident with vertices of $C_S$,
we see $G_i-v_i$ is a proper minor of $G-v$, contradicting the extremality of $\tG$.
\end{proof}

In Lemma~\ref{lemma-2conn}, we argued that if $\tG=(G,v,\nu)$ is an extremal plane near-cubic graph, then
$G-v$ is $2$-edge-connected, and thus its face containing $v$ is bounded by a cycle $C$.
Let us now argue that the graph stays $2$-edge-connected after removing $V(C)$ as well.

\begin{lemma}\label{lemma-2conn-rem}
Let $\tG=(G,v,\nu)$ be an extremal plane near-cubic graph and let $C$ be the cycle bounding the face of $G-v$ containing $v$.
The cycle $C$ is induced, no two neighbors of $v$ in $C$ are adjacent, and the graph $G-(V(C)\cup \{v\})$ is $2$-edge-connected
and has more than one vertex.
\end{lemma}
\begin{proof}
Consider a simple closed curve $c$ in the plane intersecting $G$ in two edges of $C$,
$b\le 4$ edges incident with $v$, and $r\le 1$ edges of $E(G-v)\setminus E(C)$,
where each edge is intersected at most once.
The curve $c$ separates the plane into two parts; let $A$ and $B$ be the corresponding
partition of vertices of $G$, where $v\in A$, and let $S$ be the edge cut in $G$ consisting of the edges with one end in $A$ and the
other end in $B$.  By Lemma~\ref{lemma-nosep3} applied to the edge cut in $G-v$ obtained from $S$ by removing the edges incident with $v$,
it follows that $b\le r+1$, and thus $|S|\le 3+2r\le 5$.
By Lemma~\ref{lemma-nones} we conclude that the edge cut satisfies one of the following conditions.
\begin{itemize}
\item $r=0$, $b=1$, $|S|=3$, and $B$ consists of a single vertex of $C$, or
\item $r=1$ and $G[B]$ is a subpath of $C$, or
\item $r=1$, $b=2$, and $G[B]$ is a $5$-cycle containing exactly one vertex not in $V(C)$.
\end{itemize}
If $C$ had a chord $e$, this would give a contradiction by considering a curve $c$ (with $r=0$) drawn next to the chord
so that $e\in E(G[B])$ and $b\le 3$; hence, $C$ is an induced cycle.  If two neighbors of $v$ in $C$ were adjacent,
we would obtain a contradiction by considering a curve $c$ (with $r=0$ and $b=2$) drawn around them.
If the graph $G-(V(C)\cup \{v\})$ were not connected, we would obtain a contradiction by considering a curve $c$
(with $r=0$ and $b\le 3$) chosen so that both $A$ and $B$ contain a vertex of $G-(V(C)\cup \{v\})$.
Finally, if the graph $G-(V(C)\cup \{v\})$ were not $2$-edge-connected, then we could choose $c$ so that $r=1$, $b\le 3$, and
$B$ contains a vertex of $G-(V(C)\cup \{v\})$.  But then $G[B]$ would be a $5$-cycle containing exactly one vertex not in $V(C)$
and consequently two adjacent vertices of $C$ would be neighbors of $v$, which is a contradiction.

Therefore, the graph $G-(V(C)\cup \{v\})$ is $2$-edge-connected.  Since no two neighbors of $v$ in $C$ are adjacent,
$G$ contains at least $7$ edges between $V(C)$ and $V(G)\setminus(V(C)\cup \{v\})$, and thus $G-(V(C)\cup \{v\})$ has more than one vertex.
\end{proof}

Finally, let us apply the parts (h) and (i) of Lemma~\ref{lemma-sets}.

\begin{lemma}\label{lemma-minc}
If $\tG=(G,v,\nu)$ is an extremal plane near-cubic graph, then $G$ has at least $28$ vertices.
\end{lemma}
\begin{proof}
Recall that by the definition of extremal, $d(\tG)=7$.
By Lemma~\ref{lemma-2conn}, the face of $G-v$ containing $v$ is bounded by a cycle $C$.
Let $v_1$, \ldots, $v_7$ be the neighbors of $v$ in $C$ in order.  For $i\in\{1,\ldots,7\}$,
let $P_i$ denote the subpath of $C$ from $v_i$ to $v_{i+1}$ (where $v_8=v_1$).  

By Lemma~\ref{lemma-2conn-rem}, the cycle $C$ is induced, no two neighbors of $v$ in $C$ are adjacent,
and the graph $G-(V(C)\cup \{v\})$ is $2$-edge-connected and has more than one vertex.
Hence, the face of $G-(V(C)\cup \{v\})$ containing $v$ is bounded by a cycle $C'$.
For a subgraph $G'\subseteq G$ containing $C\cup C'$, let $X(G')$ denote the set of faces of $G'$
separated from $v$ by $C'$ and let $Y(G')$ denote the set of faces of $G'$ separated from $v$ by
$C$ but not by $C'$.  
See Figure~\ref{fig-C7}(a) for an example.
For $i\in\{1,\ldots,7\}$, we say that a face $f\in X(G')$ \emph{sees $P_i$}
if there exists a face $f'\in Y(G')$ such that $f'$ is incident with an edge of $P_i$ and
the boundaries of $f$ and $f'$ share at least one edge.

If for some $i\in\{1,\ldots,7\}$, some face of $X(G)$ saw $P_i$, $P_{i+2}$, and $P_{i+4}$ (with indices
taken cyclically) then $\tG=\gamma_4(r_2(\gamma_2(\tG_1,\tG_2)),\tG_3)$ for plane near-cubic graphs
$\tG_1$, $\tG_2$, and $\tG_3$ with $d(\tG_1)=d(\tG_2)=6$ and $d(\tG_3)=7$ (see Figure~\ref{fig-C7}(b)).
Lemma~\ref{lemma-nosg} would imply $n_{\tG_j}\in K_{d(\tG_j)}$ for $j\in \{1,2,3\}$,
and by Lemma~\ref{lemma-sets}(h) and (i), we would have $n_{\tG}\in K_7$, which is a contradiction.
Hence,
\begin{equation}\label{eq-to}
\text{no face of $X(G)$ sees $P_i$, $P_{i+2}$, and $P_{i+4}$.}
\end{equation}

Let $b_1$ be the number of edges of $G$ with one end in $C$ and the other end in $C'$, let $b_2$
be the number of chords of $C'$, let $b_3$ be the number of edges with one end in $C'$ and the other end in $V(G)\setminus V(C\cup C')$,
and let $b_4$ be the number of edges of $G-v-V(C\cup C')$.  Note that $b_1\ge 7$, $b_3$ is at least three times the number of components of $G-v-V(C\cup C')$,
$|E(C)|=7+b_1$, $|E(C')|=b_1+2b_2+b_3$, and
$$|E(G)|=7+(7+b_1)+b_1+(b_1+2b_2+b_3)+b_2+b_3+b_4=14+3b_1+3b_2+2b_3+b_4.$$
A case analysis shows that since (\ref{eq-to}) holds, one of the following conditions holds:
\begin{itemize}
\item $b_1\ge 8$ and $b_2\ge 2$, or
\item $b_1\ge 8$ and $b_3\ge 3$, or
\item $b_3\ge 6$, or
\item $b_3\ge 4$ and $b_4\ge 1$.
\end{itemize}
Hence $3b_1+3b_2+2b_3+b_4\ge 30$, and thus $G$ has at least $44$ edges.  Consequently, $|V(G)|\ge (2|E(G)|-4)/3\ge 28$.
\end{proof}

As a consequence, this verifies Conjecture~\ref{conj-main} for small graphs.
\begin{corollary}\label{cor-cexsize}
Conjecture~\ref{conj-main} holds for all plane near-cubic graphs with less than $30$ vertices.
\end{corollary}
\begin{proof}
Let $\tG=(G,v,\nu)$ be a counterexample to Conjecture~\ref{conj-main}, and in particular
$n_{\tG}\not\in B'_5=K_5$.  By Lemma~\ref{lemma-growincl}, there exists a plane near-cubic graph $\tG_0=(G_0,v_0,\nu_0)$
such that $d(\tG_0)=7$, $n_{\tG_0}\not\in K_7$, and $|V(G_0)|\le |V(G)|-2$.
Hence, there exists an extremal plane near-cubic graph $\tG_1=(G_1,v_1,\nu_1)$ such that $|V(G_1)|\le |V(G_0)|$.
By Lemma~\ref{lemma-minc}, we have $|V(G_1)|\ge 28$, and thus $|V(G)|\ge 30$.
\end{proof}

\begin{figure}
\begin{center}
\begin{tikzpicture}[scale=1]
\draw[dashed] (0,0) circle(2cm);
\foreach \i in {1,...,8}{\draw (90+51.4*\i:2.1) coordinate (he\i);}
\draw[every node/.style={inner sep=1.8pt,fill,circle}]
\foreach \i in {1,...,8}{
(90+51.4*\i:1.5)  coordinate(x\i) node {} -- (he\i)
(90+25.07+51.4*\i:1.3)  coordinate (y\i) node{} -- (x\i)
(90+25.07+51.4*\i:0.8) coordinate (z\i)  node{} -- (y\i)
}
(y1)--(x2)
(y2)--(x3)
(y3)--(x4)
(y4)--(x5)
(y5)--(x6)
(y6)--(x7)
(y7)--(x1)
(z1)--(z2)--(z3)--(z4)--(z5)--(z6)--(z7)--(z1)
;
\draw(0,0) node{$X(G)$};

\foreach \i in {1,...,7}{
 \pgfmathtruncatemacro{\nexti}{\i+1}
\draw[fill=gray!70!white]
(x\nexti)--(y\nexti)--(z\nexti)--(z\i)--(y\i)--cycle
;
}
\draw[line width=2pt,color=red]
(x1)--(y1)--(x2)--(y2)--(x3)--(y3)--(x4)--(y4)--(x5)--(y5)--(x6)--(y6)--(x7)--(y7)--(x1)
(1.2,-1.2) node{$C$}
;
\draw[line width=2pt,color=blue,dotted]
(z1)--(z2)--(z3)--(z4)--(z5)--(z6)--(z7)--(z1)
(0.4,-0.9) node{$C'$}
;

\draw (0,-2.5) node{$\tilde{G}$};
\draw (0,-3.3) node{(a)};
 \pgfmathsetmacro{\MB}{360/7}
 \foreach \i in {0,...,6}{\draw (90+\i*\MB:2.3) node{\i};} 
\draw[every node/.style={inner sep=1.8pt,fill,circle}]
 \foreach \i in {1,...,7}{
 (x\i) node{}
 (y\i) node{}
 (z\i) node{}
 } 
 ;
\end{tikzpicture}
\hskip 2em
\begin{tikzpicture}[scale=1]
\draw[dashed] (0,0) circle(2cm);
\foreach \i in {1,...,7}{\draw (90+51.4*\i:2.1) coordinate (he\i);}
\draw[every node/.style={inner sep=1.8pt,fill,circle}]
\foreach \i in {1,...,7}{
(90+51.4*\i:1.5)  node (x\i){} -- (he\i)
(90+25.07+51.4*\i:1.3)  node (y\i){} -- (x\i)
(90+25.07+51.4*\i:0.8)  node (z\i){} -- (y\i)
}
(y1)--(x2)
(y2)--(x3)
(y3)--(x4)
(y4)--(x5)
(y5)--(x6)
(y6)--(x7)
(y7)--(x1)
(z1)--(z2)--(z3)--(z4)--(z5)--(z6)--(z7)--(z1)
;
\draw[densely dotted]
(150:2.3) to[bend right=40] (30:2.3)
(270:2.3) to[bend left=25] (90:0.35)
;
\draw (0,-2.5) node{$\tilde{G}$};
\draw (90:1.1) node{$\tilde{G}_{3}$};
\draw (90+2*51.4:1.1) node{$\tilde{G}_{1}$};
\draw (90-2*51.4:1.1) node{$\tilde{G}_{2}$};
 \pgfmathsetmacro{\MB}{360/7}
 \foreach \i in {0,...,6}{\draw (90+\i*\MB:2.3) node{\i};} 
 \foreach \i in {1,...,7}{\draw (90+\i*\MB+12:1.7) node{$v_{\i}$};} 
 \foreach \i in {1,...,7}{\draw (90+\i*\MB+27:1.7) node{$P_{\i}$};} 
 \draw (0,-3.3) node{(b)};
\end{tikzpicture}
\end{center}
\caption{Graph $\tG$ from Lemma~\ref{lemma-minc}. 
Edges incident to $v$ are crossing the dashed circle and $v$ is not depicted.
(a) Cycles $C$ and $C'$ are depicted by thick red and dotted blue, respectively. The gray faces  belong to $Y(G)$. The white face in the center belongs to $X(G)$.
(b) A construction of $\tilde{G}$ from $\tilde{G}_{1},\tilde{G}_{2}$ and $\tilde{G}_{3}$ is indicated by the dotted lines.}\label{fig-C7}
\end{figure}

Note that the analysis at the end of the proof of Lemma~\ref{lemma-minc} can be improved.  By a computer-assisted enumeration,
one can show that to ensure that (\ref{eq-to}) holds, $G-v$ must contain one of 38 specific graphs as a minor; the smallest are depicted in Figure~\ref{fig-28}.
Hence, every counterexample
to Conjecture~\ref{conj-main} must contain one of these 38 as a minor.
The list of these 38 graphs is available at \oururl.

\begin{figure}
\begin{center}
\begin{tikzpicture}
\draw[dashed] (0,0) circle(2cm);
\foreach \i in {1,...,7}{\draw (90+51.4*\i:2.1) coordinate (he\i);}


\draw[every node/.style={inner sep=1.8pt,fill,circle}]
\foreach \i in {1,...,7}{
(90+51.4*\i:1.5)  node (x\i){} -- (he\i)
(90+25.07+51.4*\i:1.3)  node (y\i){} -- (x\i)
(90+25.07+51.4*\i:0.8)  node (z\i){} -- (y\i)
}
(y1)--(x2)
(y2)--(x3)
(y3)--(x4)
(y4)--(x5)
(y5)--(x6)
(y6)--(x7)
(y7)--(x1)
(z1)--(z2)--(z3)--(z4)--(z5)--(z6)--(z7)--(z1)
;

\pgfmathsetmacro{\MB}{360/7}
 
\draw ($(z1)!0.5!(z2)$) node[vtx](a12){};
\draw ($(z2)!0.5!(z3)$) node[vtx](a23){};
\draw ($(z4)!0.5!(z5)$) node[vtx](a45){};
\draw ($(z6)!0.5!(z7)$) node[vtx](a67){};
\draw 
(90+2.5*\MB:0.3) node[vtx](b){}
(90-\MB:0.3) node[vtx](c){}
(a12)--(b)--(a23)
(a45)--(c)--(a67)
(b)--(c)
;
 \foreach \i in {0,...,6}{\draw (90+\i*\MB+2*\MB:2.3) node{\i};} 
\end{tikzpicture}
\hskip 3em
\begin{tikzpicture}
\draw[dashed] (0,0) circle(2cm);
\foreach \i in {1,...,7}{\draw (90+51.4*\i:2.1) coordinate (he\i);}


\draw[every node/.style={inner sep=1.8pt,fill,circle}]
\foreach \i in {1,...,7}{
(90+51.4*\i:1.5)  node (x\i){} -- (he\i)
(90+25.07+51.4*\i:1.3)  node (y\i){} -- (x\i)
(90+25.07+51.4*\i:0.8)  node (z\i){} -- (y\i)
}
(y1)--(x2)
(y2)--(x3)
(y3)--(x4)
(y4)--(x5)
(y5)--(x6)
(y6)--(x7)
(y7)--(x1)
(z1)--(z2)--(z3)--(z4)--(z5)--(z6)--(z7)--(z1)
;

\pgfmathsetmacro{\MB}{360/7}
 
\draw ($(z1)!0.5!(z2)$) node[vtx](a12){};
\draw ($(z2)!0.5!(z3)$) node[vtx](a23){};
\draw ($(z4)!0.5!(z5)$) node[vtx](a45){};
\draw ($(z6)!0.5!(z7)$) node[vtx](a67){};
\draw 
(90+1*\MB:0.3) node[vtx](b){}
(90+3.5*\MB:0.3) node[vtx](c){}
(a12)--(b)--(a67)
(a45)--(c)--(a23)
(b)--(c)
;
 \foreach \i in {0,...,6}{\draw (90+\i*\MB+2*\MB:2.3) node{\i};} 
\end{tikzpicture}

\begin{tikzpicture}
\draw[dashed] (0,0) circle(2cm);
\foreach \i in {1,...,7}{\draw (90+51.4*\i:2.1) coordinate (he\i);}


\draw[every node/.style={inner sep=1.8pt,fill,circle}]
\foreach \i in {2,3,...,7}{
(90+51.4*\i:1.5)  node (x\i){} -- (he\i)
(90+25.07+51.4*\i:1.3)  node (y\i){} -- (x\i)
(90+25.07+51.4*\i:0.8)  node (z\i){} -- (y\i)
}
(90+51.4*1:1.5)  node (x1){} -- (he1)
(90+25.07+51.4*1+12:1.3)  node (y1){}
(90+25.07+51.4*1-12:1.3)  node (y1b){} -- (x1)
(y1)--(y1b)
(90+25.07+51.4*1+18:0.8)  node (z1){} -- (y1)
(90+25.07+51.4*1-18:0.8)  node (z1b){} -- (y1b)

(y1)--(x2)
(y2)--(x3)
(y3)--(x4)
(y4)--(x5)
(y5)--(x6)
(y6)--(x7)
(y7)--(x1)
(z1)--(z2)--(z3)--(z4)--(z5)--(z6)--(z7)--(z1b)--(z1)
;

\pgfmathsetmacro{\MB}{360/7}
 
\draw ($(z3)!0.5!(z4)$) node[vtx](a34){};
\draw ($(z5)!0.5!(z6)$) node[vtx](a56){};
\draw ($(z1)!0.5!(z1b)$) node[vtx](a1){};
\draw 
(0,0) node[vtx](b){}
(a34)--(b)--(a56)
(a1)--(b)
;
 \foreach \i in {0,...,6}{\draw (90+\i*\MB+2*\MB:2.3) node{\i};} 
\end{tikzpicture}

\end{center}
\caption{The smallest minors.}\label{fig-28}
\end{figure}

\bibliographystyle{acm}
\bibliography{../data.bib}

\end{document}